\documentclass[12pt]{amsart}
\usepackage{amsmath,amssymb,amscd,graphicx, color,tensor,hyperref}
\newcommand{\R}{{ \mathbb{R}  }}

\newcommand{\bke}[1]{\left( #1 \right)}

\newcommand{\bket}[1]{\left\{ #1 \right\}}
\newcommand{\norm}[1]{\left\Vert #1 \right\Vert}
\newcommand{\abs}[1]{\left| #1 \right|}
\newcommand*{\red}{\textcolor{red}}

\newcommand*{\medcap}{\mathbin{\scalebox{1.5}{\ensuremath{\cap}}}\hspace{0.1mm}}%

\begin{document}
\bibliographystyle{plain}

\newtheorem{defn}{Definition}
\newtheorem{lemma}{Lemma}
\newtheorem{proposition}{Proposition}
\newtheorem{theorem}{Theorem}
\newtheorem{assumption}{Assumption}
\newtheorem{cor}{Corollary}
\newtheorem{remark}{Remark}
\numberwithin{equation}{section}

\newenvironment{pflem1}{{\par\noindent\bf
           Proof of Lemma~\ref{LEM1}. }}{\hfill\fbox{}\par\vspace{.2cm}}
\newenvironment{pfthm1}{{\par\noindent\bf
           Proof of Theorem~\ref{THM1}. }}{\hfill\fbox{}\par\vspace{.2cm}}
\newenvironment{pfthm2}{{\par\noindent\bf
           Proof of Theorem~\ref{THM2}. }}{\hfill\fbox{}\par\vspace{.2cm}}

\title[chemotaxis-consumption systems involving tensor-valued sensitivities]{Regular solutions of
chemotaxis-consumption systems involving tensor-valued sensitivities  and\\ Robin type boundary conditions} 

\author{Jaewook Ahn}%
\address{Department of Mathematics, Dongguk University, Seoul, 04620, Republic of Korea}%
\email{jaewookahn@dgu.ac.kr}

\author{Kyungkeun Kang}%
\address{School of Mathematics $\&$ Computing (Mathematics),  Yonsei University, Seoul 03722, Republic of Korea}%
\email{kkang@yonsei.ac.kr}

\author{Jihoon Lee}%
\address{Department of Mathematics, Chung-Ang University, Seoul 06974, Republic of Korea}%
\email{jhleepde@cau.ac.kr}

 \begin{abstract}
This paper deals with a parabolic-elliptic chemotaxis-consumption system with tensor-valued  sensitivity $S(x,n,c)$
 under no-flux boundary conditions for $n$ and Robin-type boundary conditions for $c$. The global existence of bounded classical solutions is established in dimension two under    general assumptions on tensor-valued sensitivity $S$. One of main steps is to show that $\nabla c(\cdot,t)$   becomes   tiny in $L^{2}(B_{r}(x)\cap \Omega)$ for every $x\in\overline{\Omega}$ and $t$ 
  when $r$ is sufficiently small, which seems to be of independent interest.
  On the other hand, in the case of scalar-valued sensitivity $S=\chi(x,n,c)\mathbb{I}$, there exists a bounded classical solution globally in time for two and higher dimensions provided the domain is a ball with radius $R$ and all given data are radial. The result of the radial case covers scalar-valued sensitivity $\chi$ that can be singular at $c=0$.
  \end{abstract}
  \maketitle
 
\section{Introduction}\label{sec1}
Chemotaxis-consumption systems are usually studied with scalar-valued chemotactic sensitivities where the chemotactic bacteria partially orient their movement along a gradient of a signal substance which they consume. However, according to recent modeling approaches, we do not necessarily have to assume that the chemotactic sensitivity is a scalar value. It has been suggested, based on the experimental findings~\cite{DTM05, LDW06} (see also \cite{XBO11}), to use more general, tensor-valued and spatially inhomogeneous
chemotactic sensitivity.~\cite{OH02,X15,XO09}

Taking into account tensor-valued sensitivity, in this paper, we consider the parabolic-elliptic chemotaxis-consumption system 
\begin{equation}\label{MODEL0}
 \left\{
\begin{array}{ll}
n_t =\nabla \cdot(\nabla n-  nS(x,n,c)\red{\cdot}\nabla c),\qquad\qquad& x\in\Omega,\, t>0, \vspace{1.5mm}\\
 0=\Delta c-n c,\qquad\qquad& x\in\Omega,\, t>0,
\end{array}
\right. 
\end{equation}
in a bounded smooth domain $\Omega\subset \R^{d}$, $d\ge2$, where the sensitivity $S(x,n,c)$ 
attains values in $\R^{d\times d}$.
 Here, 
 the unknowns $n$ and $c$ denote the bacterial population density and the signal concentration, respectively.
The boundary conditions posed will be
\begin{equation}\label{MODEL0BC}
(\nabla n-nS(x,n,c)\red{\cdot}\nabla c) \cdot\nu=0,\qquad \nabla c\cdot \nu=\gamma-c,\quad\quad x\in\partial\Omega,\, t>0,
\end{equation}
where $\nu$ denotes the outward unit normal vector to $\partial\Omega$.
We emphasize that the boundary condition for $c$ is of Robin type.

Homogeneous Neumann boundary conditions for $c$ have been often used in mathematical studies regarding \eqref{MODEL0} and its variants.~\cite{CL16,LS15,W15,W18,W21} However, in the original version of  \eqref{MODEL0} by Tuval et al.\cite{TC05}, 
certain non-trivial boundary conditions for $c$ are proposed
to take into account the effect of oxygen $c$ at the drop-air interface.  
Motivated by experimental observations in Tuval et al.,~\cite{TC05}    it has been  suggested in \cite{B17} (see also \cite{BL19}) to use non-homogeneous   boundary conditions  of the form
 \begin{equation}\label{RCINT}
 \nabla c\cdot \nu=(\gamma-c)g\quad\mbox{on}\quad \partial\Omega. 
 \end{equation}  
As seen in  \eqref{MODEL0BC},
 we will impose  \eqref{RCINT} with $g\equiv1$  but  results in  Theorem~\ref{THM1} are valid for more general $g$ (see Remark~\ref{RMK12}).

We compare \eqref{MODEL0}--\eqref{MODEL0BC} to the chemotaxis-consumption system with homogeneous Neumann boundary conditions 
\begin{equation}\label{PPMODEL0}
n_t  =\nabla \cdot(\nabla n-    \chi(c) n \nabla c), \quad\qquad
 c_{t}  =\Delta c-n c,\qquad\qquad  x\in\Omega,\, t>0,
\end{equation}
\begin{equation}\label{PPMODEL0N}
\quad\qquad\qquad\qquad\nabla n\cdot \nu= \nabla c\cdot \nu=0,\quad\quad\qquad\qquad\qquad x\in\partial\Omega,\, t>0.
\end{equation}
We remark that the $c$ equation should be of parabolic type since 
an elliptic approximation of the $c$ equation in 
  \eqref{PPMODEL0}--\eqref{PPMODEL0N} leads to $c\equiv0$. 
It is known that solutions of \eqref{PPMODEL0}--\eqref{PPMODEL0N}
 satisfy the energy-like inequality (see, e.g. \cite{CKL13,DLM10,W12,W14,W17})  
\begin{equation}\label{ENERGYIN}
\frac{d}{dt}\bke{ \int_{\Omega} n\log n +\frac{1}{2}\int_{\Omega} \chi(c)\frac{|\nabla c|^{2}}{c} }+\int_{\Omega}\frac{|\nabla n|^{2}}{n}+\frac{1}{4}\int_{\Omega} \frac{c}{\chi(c)}|D^{2}\rho(c) |^{2}\le0,
\end{equation}
where $\rho(c)=\int_{1}^{c} \chi(s)/s\, ds$.
The  inequality \eqref{ENERGYIN} is typically deduced via a subtle cancellation caused by nonlinear structure 
of the system \eqref{PPMODEL0}--\eqref{PPMODEL0N}. In the case of the system  \eqref{MODEL0}--\eqref{MODEL0BC},  because of presence of tensor-valued sensitivity, it is not clear whether or not such energy like inequality can be derived due to loss of the cancellation effect. As a variant of   \eqref{ENERGYIN}, we refer to \cite{BT20} for a chemotaxis-consumption-fluid system with constant sensitivity and Robin boundary condition.

As far as we know, there have been relatively few results dealing with tensor-valued sensitivity or Robin type boundary conditions. In particular, the only result in presence of both tensor-valued sensitivities and
Robin type boundary conditions we are aware of is that
 bounded weak solutions to a 3D chemotaxis-Stokes system with nonlinear cell diffusion are known to exist globally in time.\cite{TX20}

In presence of constant sensitivities and
Robin type boundary conditions,    smooth solutions to \eqref{MODEL0} are known to exist globally in time for general data and any dimension~\cite{FLM21} (see also \cite{BT20,WX20}).  However,
in presence of tensor-valued sensitivities and Neumann boundary conditions, even when $d=2$,
 smooth solutions to the fully parabolic counterpart of \eqref{MODEL0} have been found to exist globally in time only under a smallness assumption on $c$~\cite{LS15} (see also \cite{CL16}), or under additional regularizing effects such as nonlinear diffusion enhancement at large densities.~\cite{C14,WC15,WL17,WW15}
Without such additional effects,   large   data  global existence results  are so far available only for certain generalized weak solutions when $d\ge1$ in \cite{W15} and when $d=2$ in \cite{W18}. Recently,  the eventual smoothness and stabilization of certain generalized solutions are also  
investigated when $d=2$ by Winkler.~\cite{W21}

 The main motivation of the present work is to prove that smooth solutions to the two-dimensional chemotaxis-consumption system \eqref{MODEL0}--\eqref{MODEL0BC}  exist globally in time for general tensor-valued sensitivity and arbitrary large initial data.
As we mentioned earlier, it is unclear whether or not the energy-like inequality \eqref{ENERGYIN} can be derived in the case of the system \eqref{MODEL0}--\eqref{MODEL0BC}.
Instead, we   derive a series of spatially localized estimates (see Proposition~\ref{LEMGRADC}, Lemma~\ref{LEMLOCINT}, Lemma~\ref{LEMLOGL}), which will lead to the uniform in time bound of $\int_{\Omega}n  \log n $ (see Corollary~\ref{COR1}). Especially in Proposition~\ref{LEMGRADC}, which may be of independent interest, it is  shown that for arbitrary small~$\varepsilon>0$, we can find $r>0$,  independent of $x\in\overline{\Omega}$ and $t<T_{\rm max}$, such that  
 \[
\|\nabla c(\cdot,t)\|_{L^{2}(\Omega \cap B_{r}(x))}\le \varepsilon.
\] 
We  also  consider  
\eqref{MODEL0}--\eqref{MODEL0BC} with scalar-valued $S\equiv \chi \mathbb{I}_{d}$ for higher dimensions under the assumption of radial symmetry.
It will turn out that all solutions emanating from bounded radial initial data remain globally bounded, even in the case $\chi$ becomes singular at $c=0$ (see Theorem~\ref{THM2}).

 Now,   to formulate our main results, let us specify the precise problem setting.  We use the notation $\R_{+}:=(0,\infty)$.
On the tensor-valued sensitivity  $S=(S_{ij})_{i,j\in\{1,\cdots,d\}}$, we will   impose the conditions     
\begin{align}\label{S2}
\begin{aligned}[b]
&S_{ij}\in \mathcal{C}^{2}(\overline{\Omega}\times\overline{\R_{+}}\times\overline{\R_{+}}   )\quad\mbox{for all}\quad  i,j\in\{1,\cdots,d\}\quad\mbox{and}\\
 &|S (x,r,s)|+|\partial_{r}S (x,r,s)|\le S_{0}(s)\quad\mbox{for}\quad(x,r,s)\in\overline{\Omega}\times\overline{\R_{+}} \times\overline{\R_{+}} 
 \\&\mbox{with some } S_{0}\in\mathcal{C}( \overline{\R_{+}} ).
\end{aligned}
\end{align}
The boundary data $\gamma$ and the initial condition $n(\cdot,0)=n_{0}$ are assumed to satisfy   \begin{equation}\label{S3}
\gamma\in \R_{+},\qquad0\le n_{0}\in L^{\infty}(\Omega).
\end{equation}

 Our first main result is the global existence of regular solutions to the system \eqref{MODEL0}--\eqref{MODEL0BC} in two dimensions.
 \begin{theorem}\label{THM1}
Let $\Omega\subset \R^{2}$ be a bounded smooth domain.
Then, \eqref{MODEL0}--\eqref{MODEL0BC} subject to \eqref{S2}--\eqref{S3} admits a   unique non-negative   solution $(n,c)$ satisfying
\begin{align}\label{THM1SOL}
\begin{aligned}
 & n\in  \mathop{\medcap}_{p\in[1,\infty)}\mathcal{C}([0,\infty);L^{p}(\Omega))\cap \mathcal{C}^{2,1}(\overline{\Omega}\times(0,\infty))\cap L^{\infty}(0,\infty;L^{\infty}(\Omega)),
\\
& c\in \mathcal{C}^{2,0}(\overline{\Omega}\times(0,\infty))\cap  L^{\infty}(0,\infty;W^{1,\infty}(\Omega)).
\end{aligned}
\end{align}
\end{theorem}
\begin{remark}
The   results in Theorem~\ref{THM1} can be extended to higher dimensions $d\ge3$ provided   $
S(x,n,c)\equiv \chi(c)\mathbb{I}_{d}$,   
$ \chi\in\mathcal{C}^{2}(\overline{\R_{+}})$, and $\chi,\chi'\ge0$. 
Indeed,  a priori $L^{p}$-estimate shows that
\begin{align*}
\begin{aligned}
\frac{1}{p}\frac{d}{dt} &\|n^{\frac{p}{2}}\|_{L^{2}( \Omega)}^{2}+\frac{4(p-1)}{p^{2}}\|\nabla n^{\frac{p}{2}}\|_{L^{2}( \Omega)}^{2}
\\&=\frac{(p-1)}{p}\int_{\partial\Omega}\chi(c)  n^{p}(\gamma- c)-\frac{(p-1)}{p}\int_{\Omega}n^{p}\chi'(c)|\nabla c|^{2}-\int_{\Omega}n^{p+1}\chi(c)c
\\&\le\frac{(p-1)}{p}\|\chi\|_{\mathcal{C}([0,\gamma])}\gamma\|n^{\frac{p}{2}}\|_{L^{2}(\partial\Omega)}^{2},\qquad p\ge1,
\end{aligned}
\end{align*}
where we used the non-negativities of $\chi$ and $\chi'$ in the last inequality.
If we further use the trace and interpolation  inequalities and    a Moser-type iteration argument, then we can obtain a uniform-in-time bound for $n$ (see \cite{FLM21} for the case    $\chi\equiv1$).
\end{remark}
\begin{remark}\label{RMK12}
We remark that
the results in Theorem~\ref{THM1} are still valid for more general Robin boundary conditions $\nabla c\cdot \nu =(\gamma-c)g$ with $0<g\in \mathcal{C}^{1+\theta}(\partial \Omega )$ for some $\theta\in(0,1)$. This can be verified by following the same methods of proof for Theorem~\ref{THM1} and thus,  for simplicity,
all computations are performed for the case $g\equiv1$.
\end{remark}
\begin{remark}
 In Theorem~\ref{THM1}, applying the classical parabolic regularity theory\cite{LSU88} to the no-flux boundary problem
$
n_t =\nabla\cdot(\nabla n- \vec{a})$ 
for $x\in\Omega,\, t>0$ with $\vec{a}=nS\cdot \nabla c\in L^{\infty}(0,\infty;L^{\infty}(\Omega))$,
we can further have H\"older continuity of $n$ up to $t=0$ provided that 
$n_{0}$ is H\"older continuous. 
 See, e.g., 
\cite[Thm.~1.3]{PV93}.
\end{remark}
Our second main result states that in the case of scalar-valued sensitivity, 
 \eqref{MODEL0}   has a global smooth solution in two and higher dimensions   
provided the domain is a ball and all given data are radial. 
\begin{theorem}\label{THM2}
Let $\Omega=B_{R}(0)\subset \R^{d}$, $d\ge2$. Assume that \eqref{S3} holds and $n_{0}$ is radial. 
Then, \eqref{MODEL0}--\eqref{MODEL0BC}
 with the scalar sensitivity  
$
S(x,n,c)\equiv \chi(x,n,c)\mathbb{I}_{d}$, $0\le \chi\in \mathcal{C}^{2}(\overline{\Omega}\times \overline{\R_{+}}\times  \R_{+} )$, 
admits a  unique  non-negative   solution $(n,c)$ satisfying \eqref{THM1SOL}   provided that for $x,y\in\overline{\Omega}$, $r\in \overline{\R_{+}}$, $s\in\R_{+} $
\begin{align*}
\begin{aligned}
&   \chi(x,r,s)=\chi(y,r,s)\quad\mbox{if}\quad |x|=|y|\quad\mbox{and}
\\ & \chi (x,r,s)  +|\partial_{r}\chi (x,r,s)|\le \chi_{0}(s)\quad\mbox{with some }\,\, \chi_{0}\in\mathcal{C}(\R_{+}).
\end{aligned}
\end{align*} 
\end{theorem}
\begin{remark}
The proof of Theorem~\ref{THM2} mainly relies on the decay estimate of the cumulative mass distribution $Q$ defined in \eqref{DEFQ} (see Lemma~\ref{lem2}). This is crucially used to obtain the upper bound of $|\nabla c|$ and the lower bound of $c$.
\end{remark}
\begin{remark}
We emphasize that unlike the case of Theorem~\ref{THM1}, the sensitivity $\chi(\cdot,\cdot,c)$ in Theorem~\ref{THM2}  may allow singularities at $c=0$, for example, $\chi(x,n,c)=1/c$. Although $\chi$ can be singular  at $c=0$, no singularity, however, occurs  since signal concentration $c$ is turned out to be bounded below, independent of time, away from zero (see Lemma~\ref{lem4}). 
\end{remark}
The outline is as follows: in Section~\ref{SEC2}, the local existence result is established; in Section~\ref{SEC3} and Section~\ref{SEC4}, we prove Theorem~\ref{THM1} and Theorem~\ref{THM2}, respectively.   Throughout this paper, the surface area of $B_{1}(0)$  is denoted by  $\sigma_{d}$.

\section{Local existence}\label{SEC2}
In this section, we prove a local existence   result via  
 the Banach fixed point theorem.
Our local existence result reads as follows. 
\begin{lemma}\label{LEM1}
Let $\Omega\subset\R^{d}$, $d\ge2$, be a bounded smooth domain. Then, there exists a maximal time of existence, $T_{\rm max}\in(0,\infty]$, such that for $t<T_{\rm max}$, a  unique solution $(n,c)$ of \eqref{MODEL0}--\eqref{MODEL0BC} subject to \eqref{S2}--\eqref{S3} exists and satisfies
\begin{align*}
\begin{aligned}
 & n\in  \mathop{\medcap}_{p\in[1,\infty)}\mathcal{C}([0,t);L^{p}(\Omega))\cap \mathcal{C}^{2,1}(\overline{\Omega}\times(0,t)) \cap L^{\infty}(0,t; L^{\infty}(\Omega)),
\\& c\in \mathcal{C}^{2,0}(\overline{\Omega}\times(0,t)),
\end{aligned}
\end{align*}
\[
\displaystyle\int_{\Omega}n(\cdot,t)=\int_{\Omega}n_{0},\qquad   n(x,t)\ge0,\qquad  0< c(x,t)< \gamma \qquad \mbox{for}\quad x\in  \Omega.
 \]
Moreover, it holds that
\begin{equation}\label{BUCRI}
\mbox{either}\quad T_{\rm max}=\infty\quad \mbox{or}\quad \limsup_{t\nearrow T_{\rm max}}\|n(\cdot,t)\|_{L^{\infty}(\Omega)}=\infty.
\end{equation}
\end{lemma} 
 To obtain Lemma~\ref{LEM1}, we prepare
the following elementary lemma.
\begin{lemma}\label{LEM2}
Let $\Omega\subset\R^{d}$, $d\ge2$, be a bounded smooth domain, and let $p>d$. For any
$u,f\in L^{p}(\Omega)$ with $u\ge0$  and any constant  $\eta\ge0$,  the problem
\begin{equation}\label{AUXPB}
 \left\{
\begin{array}{ll}
 -\Delta v+uv=f,\quad  &x\in\Omega,\\
   \nabla v\cdot \nu+v= \eta,\quad  &x\in\partial\Omega
\end{array}
\right.
\end{equation}
admits a unique solution $v\in W^{2,p}(\Omega)\cap \mathcal{C}^{1}(\overline{\Omega})$ with the following properties:
\begin{itemize}
 \item[(i)] There exists $C=(d,\Omega,p)$ such that\[\|v\|_{W^{2,p}(\Omega)}\le C( \|uv\|_{L^{p}(\Omega)}+\|f\|_{L^{p}(\Omega)} +\eta(|\partial\Omega|^{\frac{1}{2}}+|\partial\Omega|^{\frac{1}{p}}) ).\]
\item[(ii)] If $\eta=0$, then there exists $C=(d,\Omega,p, \|u\|_{L^{p}(\Omega)})$ such that  
\[\|v\|_{W^{2,p}(\Omega)}\le C \|f\|_{L^{p}(\Omega)}.\]
\item[(iii)] If $f\equiv 0$, then $0\le v\le \eta$. 
\end{itemize}
\end{lemma}
 \begin{proof}
 To obtain the existence, we let $u$ be approximated in $L^{p}(\Omega)$ by a sequence of bounded non-negative functions $u_{l}$ and let $v_{l}\in W^{2,p}(\Omega)$ be a unique solution of the problem
\begin{equation}\label{EQVL}
 \left\{
\begin{array}{ll}
 -\Delta v_{l}+u_{l}v_{l}=f,\quad  &x\in\Omega,\\
   \nabla v_{l}\cdot \nu+v_{l}= \eta,\quad  &x\in\partial\Omega,
\end{array}
\right.
\end{equation}
which is uniquely solvable by \cite[Thm. 2.4.2.6]{G85}. Note that the elliptic regularity theory \cite[Thm. 2.3.3.6]{G85} gives that there exists $C>0$ satisfying
\[
 \|v_{l}\|_{W^{2,p}(\Omega)}\le C( \|u_{l}v_{l}\|_{L^{p}(\Omega)}+\|f\|_{L^{p}(\Omega)} +\|v_{l}\|_{W^{1-\frac{1}{p},p}(\partial\Omega)}+\eta|\partial\Omega|^{\frac{1}{p}}).
\]
We also note that    H\"older's and the Gagliardo–Nirenberg inequalities yield $C>0$ such that
\[
\|u_{l}v_{l}\|_{L^{p}(\Omega)}\le C \|u_{l} \|_{L^{p}(\Omega)} \|v_{l}\|_{L^{2}(\Omega)}^{\theta_{1}}\|v_{l}\|_{W^{2,p}(\Omega)}^{1-\theta_{1}},\qquad\theta_{1}=\frac{\frac{2}{d}-\frac{1}{p}}{\frac{2}{d}-\frac{1}{p}+\frac{1}{2}}\in(0,1),
\]
and by $W^{1,p}(\Omega) \hookrightarrow W^{1-\frac{1}{p},p}(\partial\Omega)$ and the  Gagliardo–Nirenberg inequality, there exists $C>0$ fulfilling
\[
\|v_{l}\|_{W^{1-\frac{1}{p},p}(\partial\Omega)} \le C\|v_{l}\|_{L^{2}(\Omega)}^{\theta_{2}}\|v_{l}\|_{W^{2,p}(\Omega)}^{1-\theta_{2}},\qquad\theta_{2}=\frac{\frac{1}{d}}{\frac{2}{d}-\frac{1}{p}+\frac{1}{2}}\in(0,1).
\]
Combining above estimates, after applying Young's inequality, we have that  with some $C>0$,
\begin{equation}\label{LEM2PF2}
 \|v_{l}\|_{W^{2,p}(\Omega)}\le C( \|u_{l}\|_{L^{p}(\Omega)}^{\frac{1}{\theta_{1}}} \|v_{l}\|_{L^{2}(\Omega)}+\|f\|_{L^{p}(\Omega)} + \|v_{l}\|_{L^{2}(\Omega)}+\eta|\partial\Omega|^{\frac{1}{p}}).
\end{equation}
Now, we multiply the $v_{l}$ equation by $v_{l}$, integrate over $\Omega$, and use  integration by parts, H\"older's inequality and $H^{1}(\Omega)\hookrightarrow L^{\frac{p}{p-1}}(\Omega)$ to find $C>0$ such that
\begin{align*}
\begin{aligned}
\int_{\Omega}|\nabla v_{l}|^{2}+\int_{\partial\Omega}|v_{l}|^{2}+\int_{\Omega}u_{l}v_{l}^{2}&=\int_{\Omega}fv_{l}+\int_{\partial\Omega}\eta v_{l}\\
&\le C(\|f\|_{L^{p}(\Omega)}\|v_{l}\|_{H^{1}(\Omega)})+\int_{\partial\Omega}\eta v_{l}.
\end{aligned}
\end{align*}
Then,   the Poincar\'e inequality with trace term (see e.g. \cite{BGT19})  and Young's inequality yield $C>0$ satisfying
\begin{equation}\label{LEM2PF1}
\|v_{l}\|_{H^{1}(\Omega)}\le C(\|f\|_{L^{p}(\Omega)}+\eta|\partial\Omega|^{\frac{1}{2}}).
\end{equation}
In view of 
\eqref{LEM2PF2}--\eqref{LEM2PF1},  $v_{l}$ is bounded in $W^{2,p}(\Omega)$.
The compactness of $W^{2,p}(\Omega) \hookrightarrow \mathcal{C}^{1}(\overline{\Omega})$ and the weak compactness of bounded sets in $W^{2,p}(\Omega)$ allow us to extract a subsequence  $v_{l_{j}}$ 
converging in $\mathcal{C}^{1}(\overline{\Omega})$ that converges weakly in $W^{2,p}(\Omega)$. If we take the limit in the problem for $v_{l_{j}}$, then its limit $v$ is    in $ W^{2,p}(\Omega)\cap \mathcal{C}^{1}(\overline{\Omega})$ and satisfies  \eqref{AUXPB}. This concludes the existence result. 

To obtain the uniqueness, we let $v$ and $\tilde{v}$ be two solutions. Since a simple integration by parts gives
\[
\int_{\Omega}|\nabla(v-\tilde{v}) |^{2}+\int_{\partial\Omega}|v-\tilde{v}|^{2}+\int_{\Omega}u|v-\tilde{v}|^{2}=0,
\]
we have $v\equiv \tilde{v}$.

 Repeating similar computations as above, we can find $C=C(d,\Omega,p)$ such that 
\[
\|v\|_{W^{2,p}(\Omega)}\le C \bigr{(} \|uv\|_{L^{p}(\Omega)}+\|f\|_{L^{p}(\Omega)}+\eta(|\partial\Omega|^{\frac{1}{2}} +|\partial\Omega|^{\frac{1}{p}}) \bigr{)},
\]
which yields (i).
 
  Repeating similar computations as  \eqref{LEM2PF2} and \eqref{LEM2PF1}, since $\eta=0$, there exists $C=C(d,\Omega,p)$ such that 
\[
 \|v\|_{W^{2,p}(\Omega)}\le C( \|u\|_{L^{p}(\Omega)}^{\frac{1}{\theta_{1}}}  +1)\|f\|_{L^{p}(\Omega)}.
 \]
This concludes (ii). 

 To obtain (iii), we multiply the $v$ equation by $v_{-}:=-\min\{0,v\}$, integrate over $\Omega$, and use integration by parts. Then,  we have
 \[
 \int_{\Omega}|\nabla v_{-}|^{2}+ \int_{\Omega}u |v_{-}|^{2}+\int_{\partial \Omega }|v_{-}|^{2}=-\eta\int_{\partial \Omega }v_{-}.
 \]
Since the right-hand-side is non-positive,  $v_{-}\equiv0$, namely, $v\ge0$. Using the same argument, we can   deduce $v\le \eta$. 
\end{proof}
We are now ready to prove Lemma~\ref{LEM1}.
\begin{pflem1}
We fix $p>d+2$ and let $M:=2\|n_{0}\|_{L^{p}(\Omega)}+1$. With a positive number $T<1$ to be specified below, we introduce the Banach space
\[
X_{T}:=\{f\in\mathcal{C}([0,T];L^{p}(\Omega))\,|\, \|f\|_{L^{\infty}(0,T;L^{p}(\Omega))}\le M,\,\,\,f\ge0 \,\mbox{ for }\,t\le T\}.
\]
For any given $\tilde{n}\in X_{T}$, we note from Lemma~\ref{LEM1} that   the problem 
\[
 \left\{
\begin{array}{ll}
0= \Delta c-\tilde{n}c,\quad  &x\in\Omega,\\
   \nabla c\cdot \nu= \gamma-c,\quad  &x\in\partial\Omega,
\end{array}
\right.
\]
admits a unique solution $c(\cdot,t)\in W^{2,p}(\Omega)\cap\mathcal{C}^{1}(\overline{\Omega})$ for $t\le T$ such that   $0\le c\le \gamma$. We also note, using  Lemma~\ref{LEM1}~(ii), $W^{2,p}(\Omega)\hookrightarrow \mathcal{C}^{1}(\overline{\Omega})$, and Lemma~\ref{LEM1}~(iv),  that there exists $C=C(d,\Omega,p)>0$ satisfying
\begin{equation}\label{LEM1PF1}
\|c\|_{L^{\infty}(0,T ;\mathcal{C}^{1}(\overline{\Omega}))}\le C\gamma \|\tilde{n}\|_{L^{\infty}(0,T;L^{p}(\Omega))} .
\end{equation} 
With such $c=c(\tilde{n})$, according to \cite[III.~Thm.~5.1]{LSU88}, the linear problem
\[
 \left\{
\begin{array}{ll}
n_t =\nabla \cdot(\nabla n-  nS(x,\tilde{n},c)\red{\cdot}\nabla c),\qquad\qquad& x\in\Omega,\, t>0, \vspace{1.5mm}\\
(\nabla n-nS(x,\tilde{n},c)\red{\cdot}\nabla c) \cdot\nu=0,\qquad\qquad& x\in\partial\Omega,\, t>0,\vspace{1.5mm}\\
n(x,0)=n_{0}(x), \qquad\qquad &x\in \Omega
\end{array}
\right. 
\]
has a unique weak solution $n\in \mathcal{C}([0,T];L^{2}(\Omega))\cap L^{2}(0,T;H^{1}(\Omega))$. 
If we use the weak formulation with the test function $n_{-}:=-\min\{0,n\}$, then after using \eqref{S2} and   Young's and H\"older's inequalities, we have
\begin{align*}
\begin{aligned}
\int_{\Omega}|n_{-}(\cdot,t)|^{2} &\le -2\int_{0}^{t}\int_{\Omega}|\nabla n_{-}|^{2}+2 \int_{0}^{t}\int_{\Omega}|n_{-}||S_{0}(c)||\nabla n_{-}||\nabla c|\\
&\le\frac{1}{2}\|S_{0}\|_{\mathcal{C}([0,\gamma])}^{2}\|\nabla c\|_{L^{\infty}(0,T;L^{\infty}(\Omega))}^{2}\int_{0}^{t}\int_{\Omega}|n_{-}|^{2} \quad\mbox{for}\quad t\le T 
\end{aligned}
\end{align*}
and   $n\ge0$ follows by Gr\"onwall's inequality.
Moreover,  from \cite[Thm.~VI.~6.40]{L96},  
\[
n\in L^{\infty}( 0,T ;L^{\infty}(\Omega)),
\]  and by   similar computations as above, we have 
\begin{align*}
\begin{aligned}
\int_{\Omega}&n^{p}(\cdot,t)
\\&\le \int_{\Omega}n_{0}^{p}+\frac{p(p-1)}{4} \|S_{0}\|_{\mathcal{C}([0,\gamma])}^{2}\|\nabla c\|_{L^{\infty}(0,T;L^{\infty}(\Omega))}^{2}\int_{0}^{t}\int_{\Omega}n^{p} \quad\mbox{for}\quad t\le T.
\end{aligned}
\end{align*}
Using
 Gr\"onwall's inequality and taking supremum over the time interval, it follows that
\[
\|n\|_{L^{\infty}(0,T;L^{p}(\Omega))}\le \|n_{0}\|_{L^{p}(\Omega)}\exp\bke{  \frac{(p-1)}{4}  \|S_{0}\|_{\mathcal{C}([0,\gamma])}^{2}\|\nabla c\|_{L^{\infty}(0,T;L^{\infty}(\Omega)}^{2} T }.
\]
Therefore, if we use  \eqref{LEM1PF1} and take a sufficiently small  $T$, then   the mapping $\Phi$ given by $\Phi(\tilde{n}):=n$ maps $X_{T}$ into itself.

Next, for given $\tilde{n}_{1},\tilde{n}_{2}\in X_{T}$, we denote $n_{i}=\Phi(\tilde{n}_{i})$, $c_{i}=c(\tilde{n}_{i})$ for $i=1,2$, and $\delta f=f_{1}-f_{2}$. Note that for any $t\le T$ and  $\xi\in L^{2}(0,T; W^{1,2}(\Omega))$ with $\xi_{t}\in L^{2}(0,T;L^{2}(\Omega))$, we have
\begin{align*}
\begin{aligned}
\int_{\Omega}&\delta n \xi(\cdot,t)-\int_{0}^{t}\int_{\Omega} \delta n \xi_{t} +\int_{0}^{t}\int_{\Omega}\nabla\delta n\cdot \nabla \xi
\\
&=\int_{0}^{t}\int_{\Omega}(  \delta n S(x,\tilde{n}_{1},c_{1}) \cdot \nabla c_{1}  +n_{2}  Z \cdot \nabla c_{1} +n_{2}S(x,\tilde{n}_{2},c_{2}) \cdot \nabla \delta c)\cdot \nabla \xi,
\end{aligned}
\end{align*}
where 
\[
Z= S(x,\tilde{n}_{1},c_{1})-S(x,\tilde{n}_{2},c_{1}) + S(x,\tilde{n}_{2},c_{1})-S(x,\tilde{n}_{2},c_{2}).
\]  
Note also that by the mean value theorem, \eqref{S2} and $  c\le \gamma$,  there exists $C>0$ satisfying
\[
|Z|\le C( |\delta \tilde{n}| + |\delta c|)\quad\mbox{ a.e. }\mbox{ in }\Omega\times(0,T).
\]
Along with this, if we use the above weak formulation with the test function $|\delta n|^{p-2}\delta n$, \eqref{S2}, \eqref{LEM1PF1},  and $ c\le \gamma$, then   with some $C_{1}=C_{1}(d,\Omega,p,M)>0$, we have
\begin{equation}
\begin{aligned}[b]
\frac{1}{p}\int_{\Omega}&|\delta n(\cdot,t)|^{p}+\frac{4(p-1)}{p^{2}}\int_{0}^{t}\int_{\Omega}|\nabla |\delta n|^{\frac{p}{2}}|^{2}\\
&\le C_{1}\biggr{(}\int_{0}^{t}\int_{\Omega}|\nabla |\delta n|^{\frac{p}{2}}||\delta n|^{\frac{p}{2}}+ \int_{0}^{t}\int_{\Omega}|\nabla |\delta n|^{\frac{p}{2}}||\delta n|^{\frac{p}{2}-1} n_{2} |\delta\tilde{n}| 
\\&\qquad\qquad\qquad\qquad\qquad + \int_{0}^{t}\int_{\Omega}|\nabla |\delta n|^{\frac{p}{2}}||\delta n|^{\frac{p}{2}-1}n_{2}(|\delta c|+|\nabla \delta c|)\biggr{)}.
\end{aligned}
\label{WATER}
\end{equation}
We   apply   Young's inequality to the first term on the right-hand-side above to find $C>0$ satisfying
\[
\int_{0}^{t}\int_{\Omega}|\nabla |\delta n|^{\frac{p}{2}}||\delta n|^{\frac{p}{2}} \le  \frac{p-1}{C_{1}p^{2}}\int_{0}^{t}\int_{\Omega}|\nabla |\delta n|^{\frac{p}{2}}|^{2}+C\int_{0}^{t}\int_{\Omega}  |\delta n|^{p}.
\]
Similarly, applying Young's inequality to the rightmost term, after using $W^{2,p}(\Omega)\hookrightarrow W^{1,\infty}(\Omega)$ and $n_{2}\in X_{T}$, we observe that with some  $C>0$, 
\begin{align*}
\begin{aligned}
\int_{0}^{t}&\int_{\Omega}\bke{|\nabla |\delta n|^{\frac{p}{2}}||\delta n|^{\frac{p}{2}-1}n_{2}(|\delta c|+|\nabla \delta c|)}
\\ &\le \frac{p-1}{C_{1}p^{2}}\int_{0}^{t}\int_{\Omega}|\nabla |\delta n|^{\frac{p}{2}}|^{2}+C\bke{ \int_{0}^{t}\int_{\Omega}|\delta n|^{p}+ \int_{0}^{t}\|\delta c\|_{W^{2,p}(\Omega)}^{p}}.
\end{aligned}
\end{align*}
It remains to estimate the second term on the right-hand-side of \eqref{WATER}. Note that, due to our choice of $p$, H\"older's and
the Gagliardo-Nirenberg inequalities yield $C>0$ satisfying
\begin{align*}
\begin{aligned}
&\int_{0}^{t}\int_{\Omega}|\nabla |\delta n|^{\frac{p}{2}}||\delta n|^{\frac{p}{2}-1} n_{2} |\delta\tilde{n}| 
\\&\!\le C \int_{0}^{t}\Bigr{(}    \|\nabla |\delta n|^{\frac{p}{2}}\|_{L^{2}(\Omega)}^{1+\frac{d}{p}} \|\delta n\|_{L^{p}(\Omega)}^{\frac{p-d-2}{2}}  +   \|\nabla |\delta n|^{\frac{p}{2}}\|_{L^{2}(\Omega)} \|\delta n\|_{L^{p}(\Omega)}^{\frac{p-2}{2}}     \Bigr{)} \|n_{2}\|_{L^{p}(\Omega)} \|\delta\tilde{n}\|_{L^{p}(\Omega)}.
\end{aligned}
\end{align*}
Thus, using Young's inequality and $n_{2}\in X_{T}$,   we can find $C>0$ fulfilling
\begin{align*}
\begin{aligned}
\int_{0}^{t}&\int_{\Omega}\bke{|\nabla |\delta n|^{\frac{p}{2}}||\delta n|^{\frac{p}{2}-1} n_{2} |\delta\tilde{n}| }
\\&  \le \frac{p-1}{C_{1}p^{2}}\int_{0}^{t}\int_{\Omega}|\nabla |\delta n|^{\frac{p}{2}}|^{2}+C\bke{ \int_{0}^{t}\int_{\Omega}|\delta n|^{p}+\int_{0}^{t}\|\delta\tilde{n}\|_{L^{p}(\Omega)}^{p}}.
\end{aligned}
\end{align*}
Combining the above computations, we have that with some $C>0$,
\[
\int_{\Omega}|\delta n(\cdot,t)|^{p}\le C\bke{\int_{0}^{t}\int_{\Omega}  |\delta n|^{p}+ \int_{0}^{t}(\|\delta c\|_{W^{2,p}(\Omega)}^{p}+\|\delta\tilde{n}\|_{L^{p}(\Omega)}^{p})}.
\]
Since applying Lemma~\ref{LEM2}~(ii) to the problem for $\delta c$,
\[
\left\{
\begin{array}{ll}
-\Delta \delta c+\tilde{n}_{1}\delta c =-c_{2}\delta\tilde{n},\quad  &x\in\Omega,\\
\nabla \delta c\cdot \nu+\delta c=0,\quad  &x\in\partial\Omega,
\end{array}
\right.
\]
and using $c_{2}\le \gamma$ yields $C>0$ such that
\[
\|\delta c\|_{L^{\infty}(0,T;W^{2,p}(\Omega))}\le C\gamma\|\delta\tilde{n}\|_{L^{\infty}(0,T;L^{p}(\Omega))},
\]
by Gr\"onwall's inequality, it follows  that with some $C=C(d,\Omega,p, M)>0$,
\[
\|\delta n\|_{L^{\infty}(0,T;L^{p}(\Omega))}\le CT^{\frac{1}{p}}\exp(CT)\|\delta\tilde{n}\|_{L^{\infty}(0,T;L^{p}(\Omega))}.
\]
Hence, for a sufficiently small choice of $T$, the mapping $\Phi$ becomes contraction on $X_{T}$, and by the Banach fixed point theorem, we have a unique fixed point $n=\Phi(n)$. 

Next, we consider more regularity properties of solutions. Since $n$ belongs to $\mathcal{C}([0,T];L^{2}(\Omega))\cap L^{2}(0,T;H^{1}(\Omega))\cap L^{\infty}( 0,T ;L^{\infty}(\Omega))$ and satisfies for every 
$[t_{1},t_{2}]\subset (0,T]$ and  $\xi\in W^{1,2}_{\rm loc}(0,T;L^{2}(\Omega))\cap L^{2}_{\rm loc}(0,T;W^{1,2}(\Omega))$,
\[
\int_{\Omega}n\xi(\cdot, t_{2})-\int_{t_{1}}^{t_{2}}\int_{\Omega}n\xi_{t}+\int_{t_{1}}^{t_{2}}\int_{\Omega}(\nabla n -   n S(x,n,c) \cdot \nabla c)\cdot \nabla \xi
=\int_{\Omega}n\xi(\cdot, t_{1}),
\] for any $\eta\in(0,T)$ we have $n\in\mathcal{C}^{\theta,\frac{\theta}{2}}(\overline{\Omega}\times[\eta,T])$ with some $\theta\in(0,1)$ by the classical parabolic regularity theory\cite{LSU88} (see, e.g., \cite[Thm. 1.3]{PV93}). Then, 
$c(\cdot,t)\in\mathcal{C}^{2,\theta}(\overline{\Omega})$  for $t\in[\eta,T]$ by the elliptic regularity theory [\cite[Cor. 4.41]{L13}, and moreover, since we have for $[s,t]\subset[\eta,T]$
\[
 \left\{
\begin{array}{ll}
-\Delta(c(t)-c(s))+n(t)(c(t)-c(s))=-c(s)(n(t)-n(s)),\quad  &x\in\Omega,\\
   \nabla (c(t)-c(s))\cdot \nu+(c(t)-c(s))= 0,\quad  &x\in\partial\Omega, 
\end{array}
\right.
\]
it follows by [\cite[Thm. 2.26]{L13} (see also \cite[Lem.~2.4]{FLM21}) that with some $C>0$,
\[
\|c(t)-c(s)\|_{\mathcal{C}^{2,\theta}(\overline{\Omega})}\le C\| n(t)-n(s)\|_{\mathcal{C}^{\theta}(\overline{\Omega})}\quad\mbox{for}\quad \eta\le s\le t\le T.
\]
This yields  H\"older regularity on the time variable, $c\in \mathcal{C}^{2+\theta,\frac{\theta}{2}}(\overline{\Omega}\times[\eta,T])$, and by the standard parabolic regularity theory,   $n\in\mathcal{C}^{2,1}(\overline{\Omega}\times[\eta,T])$. Since $\eta\in(0,T)$ is arbitrary, we have the desired regularity result. Note that the blow-up criteria \eqref{BUCRI} follows by the standard extension argument, the mass conservation property of $n$ is a consequence of integrating the $n$ equation, and $0<c<\gamma$ is the result of the elliptic maximum principle. 
\end{pflem1}
\begin{remark}
We remark that  Lemma~\ref{LEM1} provides local existence of the solutions in Theorem~\ref{THM1}. Moreover, Lemma~\ref{LEM1} can be also used to obtain  Theorem~\ref{THM2} since in radial case, a priori estimate shows that there exists $c_{*}>0$ such that $c\ge c_{*}$ independent of any regularization of $\chi$ keeping non-negative sign, local existence of the solutions in Theorem~\ref{THM2} is also available because singularity of $\chi$  at $c=0$ does not play any role. Since its verification is admissible,  the details are omitted.
\end{remark}

\section{Case of tensor sensitivity in two dimensions }\label{SEC3}
In this section, we prove   Theorem~\ref{THM1} via a series of spatially localized estimates.
To this end, we first establish  a  uniform-in-time smallness    of   spatially localized  $L^{2}$-norm of $\nabla c$ in the following proposition. We remark that $\nabla c$ has a uniform-in-time $L^{2}$-norm over $\Omega$
 from
\begin{align}\label{NABLACL2}
\begin{aligned} 
 \int_{\Omega}|\nabla c|^{2} &\le \int_{\Omega}|\nabla c|^{2}+\int_{\Omega}nc^{2}+\frac{1}{2}\int_{\partial\Omega}c^{2}
 \\&=\int_{\partial\Omega}\gamma c-\frac{1}{2}\int_{\partial\Omega}c^{2}
 \\&\le  \frac{1}{2} \gamma^{2}|\partial \Omega|.
\end{aligned}
\end{align}
This bound implies that $L^{2}$ norm of $\nabla c$ becomes very small in a small neighborhood of each point, but it may not be uniformly small in time. In the next proposition, we prove that it is the case, namely localized norm of $\nabla c$ can be uniformly small independent of time.

\begin{proposition}\label{LEMGRADC}
Let $\Omega\subset \R^{d}$, $d\ge2$, be a bounded smooth domain.
Let $(n,c)$ be a solution given by Lemma~\ref{LEM1}.
 For any given $\varepsilon>0$, there exists   $\delta_{\varepsilon}>0$ independent of  $q\in\overline{\Omega}$  such that
\[
\sup_{t<T_{\rm max}}\|\nabla c(\cdot,t)\|_{L^{2}(\Omega \cap B_{\delta}(q))}\le \varepsilon\quad\mbox{for}\quad  \delta\in(0,\delta_{\varepsilon}).
\]
\end{proposition}
\begin{proof}
\ Let   $\eta\in(0,e^{-1})$ and  $B_{\eta}(0)=\{ x\in\R^{d}\,|\, |x|<\eta\}$.
We   introduce  the non-negative radial   function 
\[
\psi_{\eta}(x):=
\left\{  
\begin{array}{ll}
\ln(-\ln |x| )-\ln(-\ln\eta),\qquad &x\in B_{\eta}(0)\setminus\{0\},\vspace{1mm}
\\
0,  &{\rm otherwise},
\end{array}
\right.
\]
and recall that the surface area of $B_{1}(0)$ is denoted by $\sigma_{d}$.
Direct computations show that 
\begin{align*}
\begin{aligned}
\|\psi_{\eta}\|_{L^{2}(\R^{d})}^{2}&=\sigma_{d}\int_{0}^{\eta}|\ln(-\ln r)-\ln(-\ln\eta)|^{2}r^{d-1}\,dr 
\\
&\le \sigma_{d}\int_{0}^{\eta}|\ln(-\ln r)|^{2}r^{d-1}dr
\\& =\sigma_{d}\int_{\ln\frac{1}{\eta}}^{\infty}|\ln \rho|^{2}e^{-d\rho} d\rho
\end{aligned}
\end{align*}
and 
\begin{align*}
\begin{aligned}
\|\nabla \psi_{\eta}\|_{L^{2}(\R^{d})}^{2}&=\sigma_{d}\int_{0}^{\eta}\frac{1}{|r\ln r|^{2}}r^{d-1}dr \\&=\sigma_{d}\int_{\ln\frac{1}{\eta}}^{\infty}\frac{1}{\rho^{2}}e^{-(d-2)\rho}  d\rho.
\end{aligned}
\end{align*}
Since the right hand sides above are both finite, $\psi_{\eta}\in H^{1}(\R^{d})$.  Moreover, since $\ln\frac{1}{\eta}$ tends to $\infty$ as $\eta$ approach $0$, 
\begin{equation}\label{LEMPSIV}
\|\psi_{\eta}\|_{H^{1}(\R^{d})}\rightarrow 0\quad\mbox{ as }\quad\eta\rightarrow 0. 
\end{equation}
Fix $q\in\overline{\Omega}$, and we denote $
\psi(x)=\psi_{\eta}(x-q)$ and $B_{\eta}=B_{\eta}(q)$.
If we test the $c$ equation of \eqref{MODEL0} with $c\psi^{2}$ and integrate over $\Omega$, then  integration by parts gives, due to $\psi=0$ in $(B_{\eta})^{c}$, that
\[
\int_{\Omega\cap B_{\eta}}n c^{2}\psi^{2}+\int_{\Omega\cap B_{\eta}}|\nabla  c|^{2}\psi^{2}=\int_{\partial\Omega}\nabla c\cdot\nu c\psi^{2} -2\int_{\Omega\cap B_{\eta}} \nabla  c\cdot \nabla \psi c\psi.
\]
Using  Young's inequality and  $c\le \gamma$, we  compute the rightmost term as
\[
\biggr{|}-2\int_{\Omega\cap B_{\eta}} \nabla  c\cdot \nabla \psi c\psi\biggr{|}\le \frac{1}{2}\int_{\Omega\cap B_{\eta}}|\nabla  c|^{2}\psi^{2}+2\gamma^{2}\int_{\Omega\cap B_{\eta}} |\nabla  \psi|^{2}.
\]
Next, to control the boundary term, we consider two cases.
If $B_{\eta}\subset \Omega$, then   $\psi=0$ on  $\partial\Omega$ and thus, 
\[
\int_{\partial\Omega  }\nabla c\cdot\nu c\psi^{2}=0.
\]
Otherwise, if $B_{\eta}\not\subset \Omega$, then  since $\psi=0$ in $(B_{\eta})^{c}$, we have  
\[
\int_{\partial\Omega}\nabla c\cdot\nu c\psi^{2}=\int_{\partial\Omega\cap  B_{\eta} }\nabla c\cdot\nu c\psi^{2}. 
\]
Thus, using the boundary condition and  $c\le \gamma$, we can compute   
\begin{align*}
\begin{aligned}
\biggr{|}\int_{\partial\Omega\cap  B_{\eta}  }\nabla c\cdot\nu c\psi^{2}\biggr{|}
&=\biggr{|}\int_{\partial\Omega\cap  B_{\eta} }(\gamma-c) c\psi^{2}\biggr{|}
\\
&\le \gamma^{2}\int_{\partial\Omega\cap  B_{\eta} }\psi^{2} 
\\
&\le \gamma^{2}\int_{\partial\Omega} \psi^{2}.
\end{aligned}
\end{align*}
Combining the above estimates, after using  $H^{1}(\Omega)\hookrightarrow L^{2}(\partial\Omega)$, we have that with some $C>0$ independent of $\eta$,
\[
 \int_{\Omega\cap B_{\eta}}|\nabla  c|^{2}\psi^{2}\le C\|\psi\|_{H^{1}(\R^{d})}^{2}.
\]
In view of \eqref{LEMPSIV}, there exists sufficiently small $\eta_{0}>0$ such that the right-hand-side above is less than or equal to $\varepsilon^{2}$ for   $\eta<\eta_{0}$. Moreover, since 
there exists $\delta_{0}>0$ satisfying
\[
 \psi^{2}\ge  1  \quad\mbox{a.e.}\quad\mbox{in }  B_{\delta}\quad \mbox{for}\quad \delta\in(0,\delta_{0}),
\]
we can deduce the desired result.
\end{proof}
For further local-in-space estimates, we introduce a smooth  cut-off function and its properties (see, e.g. \cite{FS16}):
\begin{lemma}\label{LEMSTEST}
Let   $\delta>0$. There is a radially decreasing function $\varphi_{ \delta }\in C_{0}^{\infty}(\R^{d})$ satisfying
\[
\varphi_{\delta}(x)=
\left\{  
\begin{array}{ll}
1,\quad &x\in B_{\frac{\delta}{2}}(0),\vspace{1mm}
\\
0,  &x\in \R^{d}\setminus B_{ \delta }(0),
\end{array}
\right.
\]
\[
0\le \varphi_{\delta}\le 1\quad\mbox{in}\quad\R^{d},
\]
and
\[
|\nabla \varphi_{\delta}|\le K\varphi_{\delta}^{ \frac{1}{2}} \quad\mbox{in}\quad\R^{d},
\]
where $K$ is a  positive constant of order $\mathcal{O}(\delta^{-1})$.
\end{lemma}

We now prepare the following lemma which is used to prove Lemma~\ref{LEMLOGL}. For computational simplicity, we use $\varphi^{3}$ as a test function.
\begin{lemma}\label{LEMLOCINT}
Let $\Omega\subset \R^{2}$  be a bounded smooth domain. Let $(n,c)$ be a solution given by Lemma~\ref{LEM1}.
Assume that  $\delta>0$  and  $\varphi_{\delta}$  is the function introduced in   Lemma~\ref{LEMSTEST}. Denote $\varphi(x)=\varphi_{\delta}(x-q)$ and $B_{\delta}=B_{\delta}(q)$  for $q\in\overline{\Omega}$.
Then, there exist two positive constants $C_{2}$ and $C_{3}$ independent of $\delta$ and $q$  such that
\begin{equation}\label{LEMLOCINT1}
\int_{\Omega}n^{2}\varphi^{3}
 \le C_{2}\bke{  \int_{\Omega} \frac{|\nabla n|^{2}}{n}\varphi^{3}  +\| \varphi^{\frac{3}{2}}  \|_{W^{1,\infty}(\R^{2})}^{2}},
\end{equation}
\begin{align}\label{LEMLOCINT2}
\begin{aligned}[b]
 \int_{\Omega}&|\nabla c|^{4}\varphi^{3}
 \\&\le C_{3}\|\nabla c\|_{L^{2}(\Omega\cap B_{\delta})}
 \bke{ \int_{\Omega} \frac{|\nabla n|^{2}}{n}\varphi^{3}     +\| \varphi^{\frac{3}{2}}  \|_{W^{1,\infty}(\R^{2})}^{2}+ \| \varphi\|_{W^{2,\frac{3}{2}}(\R^{2})}^{3}}.
\end{aligned}
\end{align}
\end{lemma}
\begin{proof}
Since the Sobolev inequality yields $C>0$ such that
\[
 \int_{\Omega}n^{2}\varphi^{3}   \le C\bke{ \|\nabla(n\varphi^{\frac{3}{2}})\|_{L^{1}(\Omega)}^{2}+\|n\varphi^{\frac{3}{2}}\|_{L^{1}(\Omega)}^{2}},
\]
after using   H\"older's inequality and $\int_{\Omega}n=\int_{\Omega}n_{0}$, we can find $C>0$,  independent of $\delta$ and $q$,   satisfying
\[
  \int_{\Omega}n^{2}\varphi^{3}   \le C\bke{ \|n_{0}\|_{L^{1}(\Omega)}  \int_{\Omega} \frac{|\nabla n|^{2}}{n}\varphi^{3} +\|n_{0}\|_{L^{1}(\Omega)}^{2}\| \varphi^{\frac{3}{2}}  \|_{W^{1,\infty}(\Omega)}^{2} }.
\]
This gives  \eqref{LEMLOCINT1}.

Next,  using  the H\"older  inequality,  direct computations, $(a+b)^{3}\le 4(a^{3}+b^{3})$ for $a,b\ge0$, and $c\le \gamma$, we note that
\begin{align}\label{LEM5_1}
\begin{aligned}[b]
 \int_{\Omega}|\nabla c|^{4}\varphi^{3}&\le   \|\nabla c\|_{L^{2}(\Omega\cap B_{\delta})}\bke{\int_{\Omega}|\nabla c|^{6}\varphi^{6}}^{\frac{1}{2}}
\\
 &= \|\nabla c\|_{L^{2}(\Omega\cap B_{\delta})}\|\nabla( c   \varphi)-c\nabla \varphi\|_{L^{6}(\Omega)}^{3}
\\
 &\le  \|\nabla c\|_{L^{2}(\Omega\cap B_{\delta})}(\|  c   \varphi\|_{W^{1,6}(\Omega)} + \|c\nabla \varphi\|_{L^{6}(\Omega)})^{3}
\\
& \le  4 \|\nabla c\|_{L^{2}(\Omega\cap B_{\delta})} ( \| c   \varphi \|_{W^{1,6}(\Omega)}^{3}+\gamma^{3}\| \nabla \varphi \|_{L^{6}(\Omega)}^{3}).
\end{aligned}
\end{align}
Since 
\[
\nabla(c\varphi)\cdot\nu=\nabla c \cdot \nu \varphi +c\nabla \varphi \cdot \nu =(\gamma-c)\varphi+c\nabla \varphi \cdot \nu  \quad\mbox{ on }\,\, \partial\Omega,
\] 
using $W^{2,\frac{3}{2}}(\Omega)\hookrightarrow W^{1,6}(\Omega)$ and 
the  elliptic regularity theory \cite[Thm. 2.3.3.6]{G85}, we can find $C>0$, independent of $\delta$ and $q$, such that
\begin{align*}
\begin{aligned}
&\|c\varphi \|_{W^{1,6}(\Omega)}
\\&\le C(  \|\Delta(c   \varphi)  \|_{L^{\frac{3}{2}}(\Omega)} +\| c   \varphi\|_{L^{\frac{3}{2}}(\Omega)} +  \| (\gamma-c)   \varphi \|_{W^{\frac{1}{3},\frac{3}{2}}(\partial\Omega)} +\| c\nabla \varphi \cdot \nu \|_{W^{\frac{1}{3},\frac{3}{2}}(\partial\Omega)} ).
\end{aligned}
\end{align*}
Using direct computations, H\"older's inequality, and $c\le\gamma$, we compute the first term on the right-hand-side above as
\begin{align*}
\begin{aligned}
\|  \Delta(c   \varphi)  \|_{L^{\frac{3}{2}}(\Omega)}&\le\| \Delta c   \varphi   \|_{L^{\frac{3}{2}}(\Omega)}+2\| \nabla c \cdot \nabla \varphi\|_{L^{\frac{3}{2}}(\Omega)}+\|  c  \Delta \varphi   \|_{L^{\frac{3}{2}}(\Omega)}
\\
&\le  \|nc     \varphi   \|_{L^{\frac{3}{2}}(\Omega)} +2\|\nabla c  \|_{L^{2}(\Omega )}\|\nabla \varphi\|_{L^{6}(\Omega)}+\gamma\|\Delta\varphi\|_{L^{\frac{3}{2}}(\Omega)}.
\end{aligned}
\end{align*}
Since  the trace inequality  and the smoothness of $\Omega$ yield $C>0$ satisfying 
\[
\| (\gamma-c) \varphi \|_{W^{\frac{1}{3},\frac{3}{2}}(\partial\Omega)}  \le C  \| (\gamma-c) \varphi \|_{W^{1,\frac{3}{2}}( \Omega)},  
\]
and
\[
\| c\nabla \varphi \cdot \nu \|_{W^{\frac{1}{3},\frac{3}{2}}(\partial\Omega)}\le C  \|c\nabla \varphi \|_{W^{1,\frac{3}{2}}( \Omega)},  
\]
using $c\le \gamma$ and H\"older's inequality, we have that with some $C>0$, independent of $\delta$ and $q$,

\begin{align*}
\begin{aligned}
  \| (\gamma-c)   \varphi \|_{W^{\frac{1}{3},\frac{3}{2}}(\partial\Omega)} +&\|   c\nabla\varphi \cdot \nu \|_{W^{\frac{1}{3},\frac{3}{2}}(\partial\Omega)} \\&\le    C(\gamma\|\varphi\|_{W^{2,\frac{3}{2}}(\Omega)}+\|\nabla c\|_{L^{2}(\Omega)}\|\varphi\|_{W^{1,6}(\Omega)}).
\end{aligned}
\end{align*}
  Note that by repeating the computations used to derive  \eqref{LEM2PF1},  we can find $C>0$ such that
\[
\|c\|_{H^{1}(\Omega)}\le  C\gamma|\partial\Omega|^{\frac{1}{2}}.
\]
Combining above estimates gives, after using $c\le \gamma$, that  with some $C>0$, independent of $\delta$ and $q$,
\[
\| c   \varphi \|_{W^{1,6}(\Omega)}\le C( \gamma\|n      \varphi   \|_{L^{\frac{3}{2}}(\Omega)} +\gamma|\partial\Omega|^{\frac{1}{2}} \|\varphi\|_{W^{1,6}(\Omega)} +\gamma\| \varphi\|_{W^{2,\frac{3}{2}}(\Omega)}   ).
\]
Plugging it into \eqref{LEM5_1}, since H\"older's inequality and $\int_{\Omega}n=\int_{\Omega}n_{0}$ imply
\begin{equation}\label{NPHI23}
\|n   \varphi   \|_{L^{\frac{3}{2}}(\Omega)} \le  \|n_{0}\|_{L^{1}(\Omega)}^{\frac{1}{3}}\bke{\int_{\Omega}n^{2}\varphi^{3}}^{\frac{1}{3}},
\end{equation}
 it follows  that there exists $C>0$, independent of $\delta$ and $q$, satisfying
\[
 \int_{\Omega}|\nabla c|^{4}\varphi^{3}\le  C \|\nabla c\|_{L^{2}(\Omega\cap B_{\delta})} \bke{   \int_{\Omega}n^{2}\varphi^{3} + \|  \varphi\|_{W^{1,6}(\Omega)}^{3} + \| \varphi\|_{W^{2,\frac{3}{2}}(\Omega)}^{3}   }.
 \]
Therefore, by \eqref{LEMLOCINT1} and $W^{2,\frac{3}{2}}(\Omega)\hookrightarrow W^{1,6}(\Omega)$, we can conclude  \eqref{LEMLOCINT2}.
\end{proof}
The spatially localized  $L\log L$-norm of $n$ is bounded uniformly in time:
\begin{lemma}\label{LEMLOGL}
Let $\Omega\subset \R^{2}$ be a bounded smooth domain. Let $(n,c)$ be a solution given by Lemma~\ref{LEM1}.
Assume that $\delta>0$  and  $\varphi_{\delta}$ is the function introduced in   Lemma~\ref{LEMSTEST}. Denote $\varphi(x)=\varphi_{\delta}(x-q)$ and $B_{\delta}=B_{\delta}(q)$ for  $q\in\overline{\Omega}$.  Then, there exist   $\delta_{*}>0$ independent of $q $ such that if $\delta<\delta_{*}$, then there exists $C=C(\delta)>0$ independent of $q $  satisfying
\[
\sup_{t<T_{\rm max}} \int_{\Omega}  n\log n(\cdot,t) \varphi^{3} \le C.
  \]
\end{lemma}
\begin{proof}
We begin by noting that due to Proposition~\ref{LEMGRADC}, there exists  $\delta_{*}>0$ independent of $q\in\overline{\Omega}$ such that 
\begin{equation}\label{DELSTAR}
\sup_{t<T_{\rm max}}(\|S_{0}\|_{\mathcal{C}([0,\gamma])}C_{2}^{\frac{1}{4}}C_{3}^{\frac{1}{4}} \|\nabla c\|_{L^{2}(\Omega\cap B_{\delta})}^{\frac{1}{4}}+\frac{1}{4}C_{3}\|\nabla c\|_{L^{2}(\Omega\cap B_{\delta})}) \le \frac{1}{5}\quad\mbox{for}\quad\delta<\delta_{*},
\end{equation}
where $S_{0}$ is the function given in \eqref{S2}, and
$C_{2}$ and $C_{3}$ are the positive numbers given in Lemma~\ref{LEMLOCINT}.

Let $\delta<\delta_{*}$. From the $n$ equation  and the no-flux   condition, we observe that
\begin{align*}
\begin{aligned}
 &\frac{d}{dt}\int_{\Omega}  n\log n \varphi^{3}-\frac{d}{dt}\int_{\Omega}n\varphi^{3} 
 \\
 &=-\int_{\Omega} \nabla(\log n \varphi^{3})\cdot [\nabla n-nS(x,n,c)\cdot\nabla c] 
\\
& =-\int_{\Omega} \frac{|\nabla n|^{2}}{n}\varphi^{3}+\int_{\Omega}\nabla n  \cdot ( S(x,n,c)\cdot\nabla c) \varphi^{3}
  \\&\qquad-3\int_{\Omega} \log n\nabla n\cdot \nabla \varphi  \varphi^{2}+3\int_{\Omega}n\log n (S(x,n,c)\cdot\nabla c  ) \cdot\nabla \varphi  \varphi^{2}.
\end{aligned}
\end{align*}
By \eqref{S2},  $c\le \gamma$, and H\"older's inequality, we have  
\begin{align*}
\begin{aligned}
&\int_{\Omega} \nabla n\cdot  ( S(x,n,c) \cdot \nabla c )   \,\varphi^{3}
\\&\le\|S_{0}\|_{\mathcal{C}([0,\gamma])}\bke{\int_{\Omega} \frac{|\nabla n|^{2}}{n}\varphi^{3}}^{\frac{1}{2}}
\bke{\int_{\Omega}n^{2}\varphi^{3}}^{\frac{1}{4}}\bke{\int_{\Omega}|\nabla c|^{4}\varphi^{3}}^{\frac{1}{4}}.
\end{aligned}
\end{align*}
 This  gives,  by Lemma~\ref{LEMLOCINT}, Young's inequality, and \eqref{NABLACL2}, that there exists   $M=M(\delta)>0$, independent of $q$,  satisfying
\begin{align*}
\begin{aligned}
&\int_{\Omega} \nabla n \cdot (S(x,n,c) \cdot \nabla c )  \,\varphi^{3}
\\&\le   \|S_{0}\|_{\mathcal{C}([0,\gamma])}C_{2}^{\frac{1}{4}}C_{3}^{\frac{1}{4}}  
\|\nabla c\|_{L^{2}(\Omega\cap B_{\delta})}^{\frac{1}{4}} \int_{\Omega} \frac{|\nabla n|^{2}}{n}\varphi^{3}  
+\frac{1}{5} \int_{\Omega} \frac{|\nabla n|^{2}}{n}\varphi^{3}  +M.
\end{aligned}
\end{align*}
Next, we use Young's inequality, $a|\log a|^{2}\le 16e^{-2}a^{\frac{3}{2}}+4e^{-2}$ for $a\ge0$, Lemma~\ref{LEMSTEST}, and \eqref{NPHI23} to compute
\begin{align*}
\begin{aligned}
-3&\int_{\Omega} \log n\nabla n\cdot \nabla \varphi  \varphi^{2}
\\& \le \frac{1}{8}\int_{\Omega} \frac{|\nabla n|^{2}}{n}\varphi^{3}+ 18\int_{\Omega} n|\log n|^{2} \varphi |\nabla \varphi|^{2}
\\&\le \frac{1}{8}\int_{\Omega} \frac{|\nabla n|^{2}}{n}\varphi^{3}+ 18\int_{\Omega} (16e^{-2}n^{\frac{3}{2}}+4e^{-2}) \varphi |\nabla \varphi|^{2}
\\&
\le  \frac{1}{8}\int_{\Omega} \frac{|\nabla n|^{2}}{n}\varphi^{3}+ 18\cdot 16e^{-2}K^{2} \int_{\Omega}n^{\frac{3}{2}}\varphi^{\frac{3}{2}}+18\cdot 4e^{-2}K^{2}\int_{\Omega}\varphi^{2}
\\&
\le  \frac{1}{8}\int_{\Omega} \frac{|\nabla n|^{2}}{n}\varphi^{3}+ 18\cdot 16 e^{-2}K^{2}  \|n_{0}\|_{L^{1}(\Omega)}^{\frac{1}{2}}\bke{\int_{\Omega}n^{2}\varphi^{3}}^{\frac{1}{2}}+18\cdot 4e^{-2}K^{2}\int_{\R^{2}}\varphi^{2}.
\end{aligned}
\end{align*}
 It follows by  Young's inequality and \eqref{LEMLOCINT1}  that with some $M=M(\delta)>0$,
\[
-3\int_{\Omega} \log n\nabla n\cdot \nabla \varphi  \varphi^{2}\le\frac{1}{5}\int_{\Omega} \frac{|\nabla n|^{2}}{n}\varphi^{3}+M.
\]
Similarly, if we use \eqref{S2}, Young's inequality, $a^{\frac{4}{3}}|\log a|^{\frac{4}{3}}\le 16e^{-\frac{4}{3}} a^{\frac{3}{2}}+e^{-\frac{4}{3}}$ for $a\ge0$, Lemma~\ref{LEMSTEST}, and \eqref{NPHI23}, then we have
\begin{align*}
\begin{aligned}
3&\int_{\Omega}n\log n  ( S(x,n,c)  \cdot \nabla c  )  \cdot\nabla \varphi  \varphi^{2}
\\&\le  3\|S_{0}\|_{\mathcal{C}([0,\gamma])}\int_{\Omega} n|\log n| |\nabla c||\nabla \varphi| \varphi^{2}
\\&
\le \frac{1}{4}\int_{\Omega}|\nabla c|^{4}\varphi^{3}+\frac{3}{4}(3\|S_{0}\|_{\mathcal{C}([0,\gamma])})^{\frac{4}{3}}\int_{\Omega} n^{\frac{4}{3}}|\log n|^{\frac{4}{3}}\varphi^{\frac{5}{3}}|\nabla \varphi|^{\frac{4}{3}}
\\&
\le \frac{1}{4}\int_{\Omega}|\nabla c|^{4}\varphi^{3}+\frac{3}{4}(3\|S_{0}\|_{\mathcal{C}([0,\gamma])})^{\frac{4}{3}}K^{\frac{4}{3}}\int_{\Omega} (16e^{-\frac{4}{3}} n^{\frac{3}{2}}+e^{-\frac{4}{3}})\varphi^{\frac{3}{2}} 
\\&
\le \frac{1}{4}\int_{\Omega}|\nabla c|^{4}\varphi^{3}+12(3\|S_{0}\|_{\mathcal{C}([0,\gamma])})^{\frac{4}{3}}K^{\frac{4}{3}} e^{-\frac{4}{3}}\|n_{0}\|_{L^{1}(\Omega)}^{\frac{1}{2}}\bke{\int_{\Omega}n^{2}\varphi^{3}}^{\frac{1}{2}}\\&\qquad +\frac{3}{4}(3\|S_{0}\|_{\mathcal{C}([0,\gamma])})^{\frac{4}{3}}K^{\frac{4}{3}}e^{-\frac{4}{3}} \int_{\R^{2}} \varphi^{\frac{3}{2}}. 
\end{aligned}
\end{align*}
Thus,  by Lemma~\ref{LEMLOCINT}  and Young's inequality, we can find   $M=M(\delta)>0$  such that
\begin{align*}
\begin{aligned}
3\int_{\Omega}n\log n & ( S(x,n,c) \cdot \nabla c  )  \cdot\nabla \varphi  \varphi^{2}
\\&\le  \frac{1}{4}C_{3} \|\nabla c\|_{L^{2}(\Omega\cap B_{\delta})} \int_{\Omega} \frac{|\nabla n|^{2}}{n}\varphi^{3}+\frac{1}{5}\int_{\Omega} \frac{|\nabla n|^{2}}{n}\varphi^{3}+M.
\end{aligned}
\end{align*}
Combining above estimates gives that with some  $M=M(\delta)>0$,
\begin{align}\label{LEM6PF1}
\begin{aligned}[b]
\frac{d}{dt}&\int_{\Omega}  n\log n \varphi^{3}-\frac{d}{dt}\int_{\Omega}n\varphi^{3} +\frac{2}{5}\int_{\Omega} \frac{|\nabla n|^{2}}{n}\varphi^{3}
\\&\le  ( \|S_{0}\|_{\mathcal{C}([0,\gamma])}C_{2}^{\frac{1}{4}}C_{3}^{\frac{1}{4}}  \|\nabla c\|_{L^{2}(\Omega\cap B_{\delta})}^{\frac{1}{4}}+ \frac{1}{4}C_{3} \|\nabla c\|_{L^{2}(\Omega\cap B_{\delta})}) \int_{\Omega} \frac{|\nabla n|^{2}}{n}\varphi^{3}  +M.
\end{aligned}
\end{align}
We note that using $a\log a\le 2e^{-1}a^{\frac{3}{2}}$ for $a\ge0$, \eqref{NPHI23},  \eqref{LEMLOCINT1}, and Young's inequality, we can find $M=M(\delta)>0$  such that
\begin{align}\label{LEM6PF2}
\begin{aligned}[b]
\int_{\Omega}  n\log n \varphi^{3}-\int_{\Omega}n\varphi^{3} &\le \int_{\Omega} 2e^{-1}n^{\frac{3}{2}} \varphi^{\frac{3}{2}}
\\&\le 2e^{-1} \|n_{0}\|_{L^{1}(\Omega)}^{\frac{1}{2}}\bke{\int_{\Omega}n^{2}\varphi^{3}}^{\frac{1}{2}} 
\\&
\le  \frac{1}{5}\int_{\Omega} \frac{|\nabla n|^{2}}{n}\varphi^{3}+M.
\end{aligned}
\end{align}
 Thus,    adding both sides of \eqref{LEM6PF1} by 
\[
\int_{\Omega}  n\log n \varphi^{3}-\int_{\Omega}n\varphi^{3} 
\]
and using \eqref{DELSTAR}, and  \eqref{LEM6PF2}, we can deduce that with some  $M=M(\delta)>0$,
\[
\frac{d}{dt}\mathcal{F}(t)+\mathcal{F}(t) \le M, 
\]
where 
\[ 
\mathcal{F}(t)=\int_{\Omega}  n\log n (\cdot,t)\varphi^{3}- \int_{\Omega}n(\cdot,t)\varphi^{3}.
\] 
Since solving this ordinary differential inequality gives
\[
\mathcal{F}(t)\le \mathcal{F}(0)e^{-t}+M(1-e^{-t}),
\]
with 
$
\int_{\Omega}n\varphi^{3}\le\int_{\Omega}n_{0}$, 
 we can conclude  the desired estimate.
\end{proof}
As a direct consequence,  $L \log L$-norm of $n$ over $\Omega$ is also bounded.
\begin{cor}\label{COR1}
Let $\Omega\subset \R^{2}$  be a bounded smooth domain.
There exists $C>0$ such that
\[
\sup_{t<T_{\rm max}} \int_{\Omega}   n \log n (\cdot,t)   \le C.
  \]
\end{cor}
\begin{proof}
Let  $\delta>0$, and let  $\varphi_{\delta}$ be a   function given  in Lemma~\ref{LEMSTEST}. Denote $\varphi(x)=\varphi_{\delta}(x-q)$ and $B_{\delta}=B_{\delta}(q)$ for  $q\in\overline{\Omega}$.
Using Lemma~\ref{LEMLOGL}, $a\log a+e^{-1}\ge  0$ for $a\ge0$, and $\varphi=1$ in $B_{\frac{\delta}{2}}$, we can find $\delta>0$ and $C>0$ both independent of $q $ such that
\[
\sup_{t<T_{\rm max}}  \int_{\Omega\cap B_{\frac{\delta}{2}}(q)}  (n\log n(\cdot,t)+e^{-1})  \le \sup_{t<T_{\rm max}}  \int_{\Omega}  (n\log n(\cdot,t)+e^{-1}) \varphi^{3} \le C.
\]
 Since the open covering $ \bigcup_{q\in\overline{\Omega}}B_{\frac{\delta}{2}}(q)$ of compact set $\overline{\Omega}$ has a finite subcovering $ \bigcup_{i=1}^{N}B_{\frac{\delta}{2}}(q_{i})$, $q_{i}\in\overline{\Omega}$, we have that with some $C>0$,
\[
\sup_{t<T_{\rm max}} \int_{\Omega }  (n\log n(\cdot,t)+e^{-1})   \le \sum_{i=1}^{N}\sup_{t<T_{\rm max}}  \int_{\Omega\cap B_{\frac{\delta}{2}}(q_{i})}  (n\log n(\cdot,t)+e^{-1})\le C.
\]
This gives the desired bound.
\end{proof}
To obtain higher integrability of $n$, we prepare the following lemma which can be seen as a generalization of  \cite[Lem.~2.4]{FS16}.
\begin{lemma}\label{LEM7}
Let $\Omega\subset \R^{2}$ be a bounded smooth domain. There exists  $C=C(\Omega)>0$ such that for any $p\ge1$, $s>1$, $\varepsilon>0$, and non-negative $f\in  \mathcal{C}^{1}(\overline{\Omega}) $, 
\begin{align*}
\begin{aligned}
 \int_{\Omega} f^{p +1}
  &\le C\frac{(p+1)^{2}}{\log s}\int_{\Omega} (f\log f+e^{-1})\int_{\Omega}f^{p -2}|\nabla f|^{2} 
  \\&\qquad\qquad+ (4C)^{1+\frac{\varepsilon}{2}}\bke{\int_{\Omega}f^{\frac{\varepsilon}{2}\frac{p +1}{1+\varepsilon}}}^{\frac{2(1+\varepsilon)}{\varepsilon}}+6s^{p  +1}|\Omega|.
\end{aligned}
\end{align*}
\end{lemma}
\begin{proof}
We recall from \cite[(2.1)--(2.5)]{FS16} that  there exists   $C=C(\Omega)>0$ such that
\begin{align*}
\begin{aligned}
\int_{\Omega}f^{p +1}\le C\frac{(p+1)^{2}}{2\log s}\int_{\Omega} (f\log f+e^{-1})\int_{\Omega}f^{p -2}|\nabla f|^{2}+2C\|  w\|_{L^{1}(\Omega)}^{2}+3s^{p  +1}|\Omega|,
\end{aligned}
\end{align*}
where 
\[
w=\max\{  f^{\frac{p +1}{2}}  - s^{\frac{p+1}{2}}   ,0\}.
\]
Using a direct computation, and H\"older's and Young's inequalities, we compute 
\begin{align*}
\begin{aligned}
\|  w\|_{L^{1}(\Omega)}^{2}\le \biggr{(}\int_{\{f>s\}}f^{\frac{p +1}{2}}\biggr{)}^{2}&\le\bke{\int_{\Omega}f^{\frac{p +1}{2+\varepsilon}}f^{\frac{\varepsilon}{2}\frac{p +1}{2+\varepsilon}}}^{2}
\\
&\le \bke{\int_{\Omega}f^{ p +1 }}^{\frac{2}{2+\varepsilon}}\bke{\int_{\Omega}f^{\frac{\varepsilon}{2}\frac{p +1}{1+\varepsilon}}}^{\frac{2(1+\varepsilon)}{2+\varepsilon}}
\\
&\le \frac{1}{4C} \int_{\Omega}f^{ p +1 } +\bke{\frac{8C}{2+\varepsilon}}^{\frac{\varepsilon}{2}} {\frac{\varepsilon}{2+\varepsilon}}\bke{\int_{\Omega}f^{\frac{\varepsilon}{2}\frac{p +1}{1+\varepsilon}}}^{\frac{2(1+\varepsilon)}{\varepsilon}}.
\end{aligned}
\end{align*}
Since
$
\bke{\frac{8C}{2+\varepsilon}}^{\frac{\varepsilon}{2}} {\frac{\varepsilon}{2+\varepsilon}}  \le (4C)^{\frac{\varepsilon}{2}}$,
 we can deduce the desired result.
\end{proof}
We are ready to prove Theorem~\ref{THM1}.  
\begin{pfthm1}
Let $(n,c)$ be a solution given by Lemma~\ref{LEM1}.
Once we have a uniform-in-time bound for $\|n\|_{L^{p}(\Omega)}$ with some $p>d=2$, then by Lemma~\ref{LEM2}~(i), $W^{2,p}(\Omega)\hookrightarrow \mathcal{C}^{1}(\overline{\Omega})$, \eqref{S2} and $c\le \gamma$, we have a uniform-in-time bound of $S(x,n,c)\nabla c$. Then, applying a Moser-type iteration argument to
\begin{align*}
\begin{aligned}
\frac{1}{p}&\frac{d}{dt}\int_{\Omega}n^{p}+\frac{4(p-1)}{p^{2}}\int_{\Omega}|\nabla n^{\frac{p}{2}}|^{2}
\\&=\frac{2(p-1)}{p}\int_{\Omega}n^{\frac{p}{2}}\nabla n^{\frac{p}{2}} \cdot(  S(x,n,c) \cdot \nabla c  ) 
\\& \le \frac{2(p-1)}{p^{2}}\int_{\Omega}|\nabla n^{\frac{p}{2}}|^{2}+\frac{(p-1)^{2}}{p }\|S(x,n,c)\nabla c\|_{L^{\infty}(\Omega)}^{2}\int_{\Omega} n^{p},\qquad p\ge1,
\end{aligned}
\end{align*}
we can find uniform-in-time bound for $n$ and   Theorem~\ref{THM1} follows by Lemma~\ref{LEM1}. Thus, it is enough to show that there exists $C>0$ satisfying
\begin{equation}\label{L3BDN}
\sup_{t<T_{\rm max}}\int_{\Omega}n^{3}(\cdot,t) \le C.
\end{equation}
 To this end, multiplying the $n$ equation by $n^{2}$ and integrating over $\Omega$, we observe that
\[
\frac{1}{3}\frac{d}{dt}\int_{\Omega}n^{3}+\frac{8}{9}\int_{\Omega} |\nabla n^{\frac{3}{2}} |^{2}=\frac{4}{3}\int_{\Omega}n^{\frac{3}{2}}\nabla n^{\frac{3}{2}}   \cdot( S(x,n,c) \cdot \nabla c ).
\]
Using \eqref{S2}, $c\le \gamma$, and H\"older's inequality, we compute the right-hand-side as
\begin{align*}
\begin{aligned}
\frac{4}{3}&\int_{\Omega}n^{\frac{3}{2}}\nabla n^{\frac{3}{2}}  \cdot( S(x,n,c) \cdot \nabla c )  \\&\le \frac{4}{3}\|S_{0}\|_{\mathcal{C}([0,\gamma])}\int_{\Omega} n^{\frac{3}{2}} |\nabla n^{\frac{3}{2}}||\nabla c|
\\&
\le \frac{4}{3}\|S_{0}\|_{\mathcal{C}([0,\gamma])}\|n\|_{L^{4}(\Omega)}^{\frac{3}{2}}\|   \nabla n^{\frac{3}{2}} \|_{L^{2}(\Omega)}\|c\|_{W^{1,8}(\Omega)}.
\end{aligned}
\end{align*}
Then, we  use the Gagliardo-Nirenberg interpolation inequality,
\[
\|f\|_{W^{1,8}(\Omega)} \le C(\|f\|_{L^{\infty}(\Omega)}^{\frac{1}{2}}\|f\|_{W^{2,4}(\Omega)}^{\frac{1}{2}}+\|f\|_{L^{\infty}(\Omega)})\quad\mbox{ for all }\quad f\in \mathcal{C}^{2}(\overline{\Omega}),
\]
Lemma~\ref{LEM2}~(i), and   $c\le \gamma$ to find  $C>0$ such that
\[
\|c\|_{W^{1,8}(\Omega)}\le C (\|n\|_{L^{4}(\Omega)}^{\frac{1}{2}}+1).
\]
Since Lemma~\ref{LEM7} with $(f,p,\varepsilon)=(n,3,1)$ and Corollary~\ref{COR1} yield $C>0$ independent of $s>1$ satisfying
\[
\int_{\Omega} n^{4}
 \le  \frac{C}{\log s} \int_{\Omega} |\nabla n^{\frac{3}{2}}|^{2} + (4C)^{\frac{3}{2}}\bke{\int_{\Omega}n_{0}}^{4}+6s^{4}|\Omega|,
\]
combining above estimates, after using Young's inequality, we can find $C >0$  independent of $s>1$ such that
\begin{align*}
\begin{aligned}
\frac{1}{3}&\frac{d}{dt}\int_{\Omega}n^{3}+\frac{8}{9}\int_{\Omega} |\nabla n^{\frac{3}{2}} |^{2}
\\&
\le  \frac{4}{3}\|S_{0}\|_{\mathcal{C}([0,\gamma])}\|n\|_{L^{4}(\Omega)}^{\frac{3}{2}}\|   \nabla n^{\frac{3}{2}} \|_{L^{2}(\Omega)}\|c\|_{W^{1,8}(\Omega)}
 \\&\le C \bke{ \frac{1}{\sqrt{\log s}}\|   \nabla n^{\frac{3}{2}} \|_{L^{2}(\Omega)}+1+s^{2} }\|   \nabla n^{\frac{3}{2}} \|_{L^{2}(\Omega)}.
\end{aligned}
\end{align*}
If we take sufficiently large $s$ and use Young's inequality, then with some $C>0$,
\[
 \frac{d}{dt}\int_{\Omega}n^{3}+ \int_{\Omega} |\nabla n^{\frac{3}{2}} |^{2}\le C.
\]
This implies, by the Gagliardo-Nirenberg type inequality,
\[
\|f\|_{L^{3}(\Omega)}^{3}\le \|\nabla f^{\frac{3}{2}}\|_{L^{2}(\Omega)}^{2}+C\|f\|_{L^{1}(\Omega)}^{3}\quad\mbox{for all}\quad f\in\mathcal{C}^{1}(\overline{\Omega}),
\] 
and $\int_{\Omega}n=\int_{\Omega}n_{0}$, that with some $C>0$,
\[
 \frac{d}{dt}\int_{\Omega}n^{3}+\int_{\Omega}n^{3}\le C.
\]
Therefore, we can deduce \eqref{L3BDN}.
\end{pfthm1} 

\section{Case of scalar sensitivity in general dimensions}\label{SEC4}
Throughout this section, let
$
\Omega=B_{R}(0)=\{x\in\R^{d} \,|\,r=|x|<R \} 
$, and $n_{0}$ 
be radial. 

In the radially symmetric setting, two equations of \eqref{MODEL0}   can be
written as
\[
n_{t}= r^{1-d} \bke{r^{d-1} n_{r}}_{r}- r^{1-d} \bke{ r^{d-1}n\chi(r,n,c)c_{r}   }_{r},\quad  r^{1-d} \bke{r^{d-1} c_{r}}_{r}=nc.
\]
Thus,  the cumulative mass distribution $Q$ defined by
\begin{equation}\label{DEFQ}
Q(r,t):=\int_{B_{r}(0)}n(x,t) \,dx=\sigma_{d}\int_{0}^{r}\rho^{d-1} n(\rho,t)\,d\rho
\end{equation}
satisfies  
\begin{equation}\label{MODEL_R}
  Q_{t}=r^{d-1}\bke{ r^{1-d} Q_{r}  }_{r}-Q_{r}\chi(r,n,c)c_{r},\qquad  r<R, t<T_{\rm max}.
\end{equation}
We note that    
\[
Q(R,t)=\|n_{0}\|_{L^{1}(\Omega)},
\]
and  
\begin{equation}\label{diffQC}
Q_{r} \ge0, \qquad c_{r} =r^{1-d}\int_{0}^{r}\rho^{d-1}nc\,d\rho \ge0,\qquad  r<R, t<T_{\rm max}.
\end{equation}
The non-negativities \eqref{diffQC} and   $\chi\ge0$ yield an upper bound for   
 $Q$ stated below.
\begin{lemma}\label{lem2}
Let $Q$ be the cumulative mass distribution defined in \eqref{DEFQ}.
Then, there exists  $M_{0}=M_{0}(d,R,\norm{n_{0}}_{L^{1}(\Omega)},\norm{n_{0}}_{L^{\infty}(\Omega)})\ge0$ such that
\[
Q(r,t)\le M_{0}r^{d}\quad\mbox{for}\quad r<R,   t < T_{\rm max}.
\]
\end{lemma}
\begin{proof}
We use a comparison argument.
 Due to \eqref{diffQC} and  $\chi\ge0$, it follows  from \eqref{MODEL_R}  that
\[
Q_{t}\le r^{d-1}\bke{r^{1-d} Q_{r}  }_{r}.
\]
Define
\[
M_{0}: =\max\bket{\frac{1}{R^{d}}\norm{n_{0}}_{L^{1}(\Omega)},\,\,    \frac{\sigma_{d}}{d}\norm{n_{0}}_{L^{\infty}(\Omega)} },
\]   
and 
\[
W(r):=M_{0} r^{d}.
\]
Then, $Q(R,t)\le W(R)$, $Q(r,0)\le W(r)$, and
\[
0=  r^{d-1}\bke{r^{1-d}W_{r}  }_{r}.
\]
Let $\varepsilon>0$ be given, and we now show that $F$ defined by
\[
F(r,t):=(Q(r,t)-W(r,t))\exp(- t)
\]
can not attain value $\varepsilon$ as long as solution exists. Note that $F(0,t)=0$, $F(R,t)\le0$, $F(r,0)\le 0$, and
\begin{align*}
\begin{aligned}
F_{t}&=(Q_{t}-W_{t})\exp(-t)-F
\\&\le  r^{d-1}\bke{r^{1-d}F_{r}  }_{r}-  F
\\&=F_{rr}+(1-d)r^{-1} F_{r}-  F.
\end{aligned}
\end{align*}
Assume to the contrary that $F(r_{1},t_{1})=\varepsilon$ for the first time $t_{1}<T_{\rm max}$. Then,  $r_{1}\not=0$ or $R$ and   
\[
0\le F_{t}(r_{1},t_{1}),\qquad F_{rr}(r_{1},t_{1})\le 0,
\]
\[
(1-d)r_{1}^{-1} F_{r}(r_{1},t_{1})=0,\qquad - F(r_{1},t_{1})=- \varepsilon<0,
\]
which leads to a contradiction. Since $\varepsilon>0$ is arbitrary, $F\le0$ and the desired bound follows.
\end{proof}
Due to Lemma~\ref{lem2},   for each   $t<T_{\rm max}$,  there exist  
 a radius  $r_{t}\in[0,R]$ and a   number   $m_{0}>0$ satisfying $c(r_{t},t)\ge m_{0}$:
\begin{lemma}\label{lem3}
Let $(n,c)$ be the solution given by Lemma~\ref{LEM1}, and let $M_{0}$ be a number given in Lemma~\ref{lem2}. Then, for each   $t<T_{\rm max}$,  there exists a radius $r_{t}\in[0,R]$ such that 
\begin{equation}\label{LEM3_1}
c(r_{t},t)\ge m_{0}:=\frac{\gamma}{2}\bke{ \frac{M_{0}R}{\sigma_{d}}+1 }^{-1}.
\end{equation}
\end{lemma}
\begin{proof}
Suppose that \eqref{LEM3_1} is false. Then, there exists $T<T_{\rm max}$  such that
\[
c(r,T)<  m_{0} \quad\mbox{for all}\quad r\in[0,R].
\]
Fix $t=T$. Using  the $c$ equation and Lemma~\ref{lem2}, we can estimate   
\[
c_{r}=r^{1-d}\int_{0}^{r}\rho^{d-1}nc\,d\rho < \frac{M_{0}m_{0}}{\sigma_{d}}r \quad\mbox{for all}\quad r\in(0,R].
\] 
If we take $r=R$, then from the boundary condition and $c< m_{0}$, we have
\[
\gamma-m_{0} < \gamma-c(R)=c_{r}(R)<  \frac{M_{0}m_{0}}{\sigma_{d}}R.
\]
This leads to a contradiction because  $m_{0}<\gamma\bke{ \frac{M_{0}R}{\sigma_{d}}+1 }^{-1}$.
\end{proof}
As a consequence,   $c$ has the lower bound which is   uniform in space and time:
\begin{lemma}\label{lem4}
Let $(n,c)$ be the solution given by Lemma~\ref{LEM1}, and let  $M_{0}$ and $m_{0}$ be  numbers given in Lemma~\ref{lem2} and  Lemma~\ref{lem3}, respectively.
Then, it holds that
\[
\min_{r\in[0,R]} c(r,t)\ge   c_{*}:= m_{0}\exp\bke{-\frac{1}{2}\frac{M_{0}R^{2}}{\sigma_{d}}} \quad\mbox{for}\quad  t < T_{\rm max}.
\]
\end{lemma}
 \begin{proof}
Note that by Lemma~\ref{LEM1},  
\[
c>0,
\] and  by  $c_{r}\ge0$ in  \eqref{diffQC},
\[
\min_{r\in[0,R]}c(r,t)=c(0,t).
\]  Since for each   $t<T_{\rm max}$, there exists $r_{t}\in[0,R]$ satisfying \eqref{LEM3_1},
in view of  Lemma~\ref{lem2} and
 \[
 r^{1-d} \bke{r^{d-1} (\log c)_{r}}_{r}= \Delta \log c =n-\abs{\nabla \log c}^{2}\le n, 
 \]
we have that
 \[
  (\log c)_{r}=  r^{1-d}  \int_{0}^{r}\bke{\rho^{d-1} (\log c)_{\rho}}_{\rho}\,d\rho \le  r^{1-d}\int_{0}^{r}\rho^{d-1}n \,d\rho\le \frac{M_{0}}{\sigma_{d}}r.
 \] 
If we integrate it from $0$ to $r_{t}$, then   
\[
\log \frac{c(r_{t},t)}{c(0,t)} \le \int_{0}^{r_{t}}\frac{M_{0}}{\sigma_{d}}\rho \,d\rho=\frac{M_{0}}{2\sigma_{d}}r_{t}^{2}.
\]
Therefore, using  \eqref{LEM3_1} and $r_{t}\le R$, we can deduce the desired result.
\end{proof}
We are ready to prove Theorem~\ref{THM2}.  
\begin{pfthm2}
Let $(n,c)$ be the solution given by Lemma~\ref{LEM1}. 
From \eqref{diffQC} and $Q(r,t)\le M_{0}r^{d}$ in Lemma~\ref{lem2}, we have 
\[
0\le c_{r}=r^{1-d}\int_{0}^{r}\rho^{d-1}nc\,d\rho \le \frac{\gamma M_{0} }{\sigma_{d}}r\quad\mbox{for}\quad r\in(0,R].
\] 
Thus, $\nabla c$ has uniform-in-time pointwise bounds.
Moreover,  by  $c_{*}\le c\le\gamma$   and $\chi(x,n,c)\le \chi_{0}(c)\in\mathcal{C}(\R_{+})$, we have that  $\chi(x,n,c) $ is bounded uniformly in time.  
Then, the standard parabolic regularity theory gives a uniform-in-time bound for $n$. This concludes Theorem~\ref{THM2} from Lemma~\ref{LEM1}.  
\end{pfthm2}  

\section*{Acknowledgment}
The authors express sincere gratitude to the anonymous referees for their helpful remarks and their careful reading of our manuscript. J. Ahn was supported by National Research Foundation (NRF) of Korea (Grant No. NRF-2021R1F1A1064209). K. Kang was supported by NRF-2019R1A2C1084685. 
 J. Lee was supported by Samsung Science and Technology Foundation under Project No. SSTF-BA1701-05.

\end{document}